\documentclass[reqno]{amsart}

\usepackage{hyperref}
\usepackage{breakurl}
\usepackage{amssymb}
\usepackage{amsthm}
\usepackage{mathtools}
\usepackage{enumitem}

\usepackage{xcolor}
\hypersetup{
    colorlinks,
    linkcolor={red!50!black},
    citecolor={blue!50!black},
    urlcolor={blue!80!black}
}

\usepackage{multirow}

\newcommand{\Real}{\operatorname{Re}}
\newcommand{\Imag}{\operatorname{Im}}

\newcommand{\ZZ}{\mathbb{Z}}
\newcommand{\QQ}{\mathbb{Q}}

\newcommand{\RR}{\mathbb{R}}
\newcommand{\CC}{\mathbb{C}}

\newcommand{\QQbar}{\overline{\mathbb{Q}}}

\newcommand{\be}{\begin{equation}}
\newcommand{\ee}{\end{equation}}

\newtheorem{theorem}{Theorem}
\newtheorem{lemma}[theorem]{Lemma}

\theoremstyle{definition}

\newtheorem{algorithm}[theorem]{Algorithm}

\makeatletter
\@namedef{subjclassname@2020}{%
  \textup{2020} Mathematics Subject Classification}
\makeatother

\begin{document}

\title{Rapid computation of special values of Dirichlet L-functions}

\author{Fredrik Johansson}
\address{Inria Bordeaux, 33400 Talence, France}
\email{fredrik.johansson@gmail.com}



\begin{abstract}
We consider computing the Riemann zeta function $\zeta(s)$
and Dirichlet $L$-functions $L(s,\chi)$
to $p$-bit accuracy for large $p$.
Using the approximate functional equation together
with asymptotically fast computation of the incomplete gamma function,
we observe that $p^{3/2+o(1)}$ bit complexity can be
achieved if $s$ is an algebraic number of fixed degree
and with algebraic height bounded by $O(p)$.
This is an improvement over the $p^{2+o(1)}$ complexity of previously published algorithms
and yields, among other things, $p^{3/2+o(1)}$ complexity algorithms for Stieltjes
constants and $n^{3/2+o(1)}$ complexity 
algorithms for computing the $n$th Bernoulli number or the $n$th Euler number exactly.
\end{abstract}

\maketitle


\section{Introduction}

Let $\chi$ be a Dirichlet character modulo $q \ge 1$.
The associated Dirichlet $L$-function is
the analytic continuation of
\be
L(s,\chi) = \sum_{n=1} \frac{\chi(n)}{n^s}, \quad \quad \Real(s) > 1
\label{eq:lseries}
\ee
to $s \in \CC$ with the possible exception of a pole at $s = 1$.
The Riemann zeta function $\zeta(s)$ is
the Dirichlet $L$-function corresponding to the trivial character $\chi(n) = 1$,
which is the unique character modulo $q = 1$.

If $\chi$ is a primitive character,
then the function $L(s,\chi)$ is represented in the entire complex plane
by a convergent expansion,
the \emph{approximate functional equation} \cite[Theorem~7.3]{Cohen2019}
\begin{equation}
\begin{aligned}
\Gamma\!\left(\frac{s+\delta}{2}\right) \!L(s,\chi) &= 
\delta_{q,1}\pi^{s/2}\!\left(\frac{\alpha^{(s-1)/2}}{s-1}-\frac{\alpha^{s/2}}{s}\right)
+ \sum_{n=1}^{\infty} \frac{\chi(n)}{n^s} \, \Gamma\!\left(\frac{s+\delta}{2}, \frac{\pi n^2 \alpha}{q}\right) \\
& + \omega \left(\frac{\pi}{q}\right)^{\!\!s-1/2} \sum_{n=1}^{\infty} \frac{\overline{\chi}(n)}{n^{1-s}} \, \Gamma\!\left(\frac{1-s+\delta}{2}, \frac{\pi n^2}{ \alpha q}\right)
\end{aligned}
\label{eq:fe}
\end{equation}
where $\delta \in \{0, 1\}$ is the parity $\chi(-1) = (-1)^\delta$,
$\omega$ is the \emph{root number} of $\chi$,
which satisfies $|\omega| \le 1$,
and $\Gamma(a,z) = \int_z^{\infty} t^{a-1} e^{-t} dt$ is
the incomplete gamma function.
The quantity $\alpha$ is a free positive parameter;
we may take $\alpha = 1$ to balance the rate of convergence
of both series.
If $\chi$ is not primitive, we can decompose $L(s,\chi)$
in terms of primitive functions.


The expansion \eqref{eq:fe} is useful for high-precision
computation of $L(s,\chi)$ due to the super-exponential
decay $\Gamma(a, C n^2) \approx \exp(-C n^2)$ of the incomplete gamma
functions. We need only $O(p^{1/2})$ terms
for a desired bit precision $p$,
which should be contrasted with
Euler-Maclaurin summation~\cite[\S 4.2]{belabas2021numerical} \cite{Johansson2014hurwitz}
and similar methods which require $O(p)$ terms.
(An equally important advantage of \eqref{eq:fe}
is that we only need $O(q^{1/2})$
terms as a function of the modulus $q$.)

The drawback of \eqref{eq:fe} is that we have
to compute the nonelementary incomplete gamma functions.
Our goal is to study this problem
with attention to the bit complexity when $p \to \infty$.
We use ``time'' and ``bit operations'' synonymously, and
recall that floating-point numbers with $p$-bit precision
can be multiplied in time $O(p \log p) = p^{1+o(1)}$~\cite{Harvey2021}.
Using Euler-Maclaurin summation, for example, it is easy to show that
we can compute $L(s,\chi)$ or any of its $s$-derivatives $L^{(j)}(s,\chi)$
to $p$-bit accuracy in time $p^{2+o(1)}$
for fixed $s$, $\chi$ and $j$.
Our main observation is the following improved complexity bound
for special values $s$.


\begin{theorem}
Let $\chi$ be a fixed Dirichlet character,
and let $s \in \overline{\QQ}$ be an algebraic number of fixed degree
such that the minimal polynomial of $s$ over $\mathbb{Z}$
has coefficients bounded in absolute value by $O(p)$.
Then, for any fixed $j \ge 0$, the value
$L^{(j)}(s,\chi)$
can be approximated with absolute error less than $2^{-p}$ in time $p^{3/2+o(1)}$ using $p^{1+o(1)}$ space.
(When $s = 1$ and this point is a pole, the corresponding Laurent series coefficient is computed.)
\label{thm:complexity}
\end{theorem}

\begin{proof}
The height condition implies that $|s| = O(p)$.
Since any terms and prefactors appearing in \eqref{eq:fe}
and in the asymptotics of the incomplete gamma function are
bounded by $\exp(|s|^{1+o(1)})$,
it is sufficient to truncate both infinite series to $N = p^{1/2+o(1)}$
terms and approximate the terms to $p^{1+o(1)}$ bits.

The function $y(z) = \Gamma(a, z)$ is holonomic,
satisfying a linear differential equation
$A y = 0$ with $A \in \mathbb{Z}[a,z,\tfrac{d}{dz}]$.
We evaluate $y$ (for two values of the
parameter $a$) at $N$ points $z$.
We can compute each such function value
in $p^{1+o(1)}$ bit operations and $p^{1+o(1)}$ space
using the \emph{bit-burst algorithm}~\cite{david1990computer},
employing arithmetic in the number field $\QQ(s)$.
This bound holds uniformly for the required values of $a$ and $z$,
by the same argument as in \cite[Corollary 1]{mezzarobba2012note},
using the facts that $|z| \le p^{1+o(1)}$
and that $A$ as well as the defining polynomial of $\mathbb{Q}(s)$
have fixed degree and coefficients of bit size $O(\log p)$.

Using standard methods,
the remaining operations (evaluation of Dirichlet characters,
the gamma function, and elementary functions) fall within the
same complexity bounds.

For derivatives $L^{(j)}(s,\chi)$, and at removable singularities,
the equivalent operations
can be carried out using arithmetic on truncated formal power series.
\end{proof}

We will provide additional details below.
The only interesting point in the proof of Theorem~\ref{thm:complexity}
is the use of the bit-burst algorithm
instead of naive summation, which allows us to compute
the function $\Gamma(a,z)$ in quasilinear rather
than quadratic time.
The bit-burst algorithm for holonomic functions
has been known since the 1980s~\cite{david1990computer,vdH:hol,van2001fast,Mezzarobba2011,mezzarobba2012note},
and since the 1970s in special cases~\cite{brent1976complexity},
yet we are not aware of a correct complexity bound
of this kind in the literature for Dirichlet $L$-functions
or even for the special case of the Riemann zeta function.

The use of the approximate functional equation for $L$-function
computation
has been studied in detail
by Rubinstein~\cite{rubinstein1998evidence}
and several other authors~\cite{dokchitser2004computing,booker2006artin,molin2010integration,belabas2021numerical}.
These works do not mention the bit-burst algorithm or address
the bit complexity for large $p$, instead focusing
on parameters relevant for numerical testing of the generalized Riemann hypothesis,
namely with fixed $p$ and with large $|\Imag(s)|$ and/or large $q$.
For large $|\Imag(s)|$, one must use a ``smoothed'' version of \eqref{eq:fe}
to avoid exponentially large cancellation, or Riemann-Siegel
type expansions; the method in Theorem~\ref{thm:complexity}
is not competitive in this setting, where the best methods
achieve $O(|\Imag(s)|^{1/2})$ or lower complexity for a fixed level of accuracy.

Borwein, Bradley and Crandall \cite{BorweinBradleyCrandall2000}
and Crandall~\cite{crandall2012unified,bailey2015crandall} discuss
the approximate functional equation in the context of high-precision
zeta function computation,
but do not mention the bit-burst algorithm or give a complexity
bound of this type. Crandall~\cite{crandall2012unified} writes that the incomplete gamma function
can be computed using $p^{1+o(1)}$ ``operations'', but this
is referring to full-precision arithmetic operations,
which would give us $p^{2+o(1)}$ bit complexity for $\Gamma(a,z)$
and $p^{5/2+o(1)}$ bit complexity for $L$-functions.
A refined algorithm that achieves $p^{2+o(1)}$ bit complexity is described in \cite[\S 7]{BorweinBradleyCrandall2000};
see \S \ref{sect:transcendental} below.

In a 1988 paper, Borwein and Borwein~\cite{Borwein1988} claim
that $\zeta(s)$ can be computed in time $p^{1+o(1)}$
if $s$ is a fixed rational number, and in time $p^{3/2+o(1)}$
if $s$ is a fixed generic (computable) complex number.
No explicit algorithm is given to justify these claims: the authors
simply write ``we truncate both the integral and the sum''
with reference to the formula
\be
\zeta(s) \Gamma\left(\frac{s}{2}\right) \pi^{-s/2} - \frac{1}{s(s-1)} = \int_{1}^{\infty} \frac{t^{(1-s)/2} + t^{s/2}}{t} \sum_{n=1}^{\infty} e^{-n^2 \pi t} \, dt
\label{eq:fetheta}
\ee
which, apart from minor differences in notation, is the approximate functional equation for
$\zeta(s)$ in the form originally derived by Riemann~\cite{riemann1859ueber}
(we obtain the series in incomplete gamma functions by integrating term by term).

Both bounds claimed by the Borweins are a factor $p^{1/2}$ better than all methods
known to this author.
Lacking evidence to the contrary, we believe that the Borweins had in mind some
combination of the
algorithms that will be described below
and that their complexity analysis was erroneous.
Our goal with this article is therefore in part to correct the record.\footnote{We mention that J.~Borwein
coauthored the 2000 survey paper~\cite{BorweinBradleyCrandall2000}
on $\zeta(s)$ computation, which discusses the
approximate functional equation prominently but does not mention the
claims from 1988. This omission suggests that the Borweins were aware
of the error (but perhaps did not consider it
important enough to publish a correction).
The other complexity results in~\cite{Borwein1988},
for instance concerning $\Gamma(s)$, are correct.}

The rest of this paper is structured as follows:
\S \ref{sect:applications} discusses some consequences
of Theorem~\ref{thm:complexity}, \S \ref{sect:algorithm} gives a more
detailed description of the algorithm,
and \S \ref{sect:implementation} reports
implementation results.
Finally, \S \ref{sect:transcendental} discusses alternative
algorithms for use when $s$ is not algebraic.



\section{Applications and generalizations}

\label{sect:applications}

The immediate application of Theorem~\ref{thm:complexity} is
that $p^{3/2+o(1)}$ can be a significant improvement
over $p^{2+o(1)}$ for numerical evaluation
to tens of thousands of digits.
Such computations are not exclusively done to test algorithms; for
example, integer relation searches
employing 50,000-digit precision have been
successful in discovering new identities
involving special values of $L$-functions~\cite{Bailey2001}.

\subsection{Values at integers}

The most famous special values, $\zeta(n)$ and $L(n,\chi)$ with $n \in \ZZ$,
can be expressed in terms of
logarithmic derivatives
of the gamma function at rational points,
and can consequently be computed in
quasilinear time $p^{1+o(1)}$
using binary splitting~\cite{Karatsuba1998,johansson2021arbitrary}.

Alternatively (and often more efficiently),
binary splitting can be applied
directly to convergence-accelerated series for the Riemann zeta function~\cite[\S 4.7]{Johansson2014thesis}
or hypergeometric series for particular values such as
\be
\zeta(3) = \frac{5}{2} \sum_{n=1}^{\infty} (-1)^{n+1} \frac{(n!)^2}{n^3 (2n)!}
\ee
which have been used to compute billions of digits \cite{sebah2003zeta,Yee2021}.

However, these quasilinearity results all assume that $n = p^{o(1)}$.
For example, if $n$ and $p$ are proportional (or proportional up to logarithmic factors),
then the complexity degenerates to $p^{2+o(1)}$ or worse.
The complexity is also softly quadratic
with $n \propto p$ if we compute $L(s,\chi)$ directly using the $L$-series \eqref{eq:lseries} or the corresponding
Euler product.
The $p^{3/2+o(1)}$
complexity of Theorem~\ref{thm:complexity} is then an improvement over previous algorithms.

\subsection{Bernoulli and Euler numbers}

The Bernoulli numbers and Euler numbers are the rational
numbers and integers respectively defined by
\be
\frac{x}{e^x-1} = \sum_{n=0}^{\infty} \frac{B_n}{n!} x^n, \quad
\frac{1}{\cosh(x)} = \sum_{n=0}^{\infty} \frac{E_n}{n!} x^n.
\ee
The odd-index values are trivial, while
the even-index values can be expressed in terms of Dirichlet $L$-functions
as
\begin{equation}
B_{2n} = (-1)^{n+1} \frac{2 (2n)!}{(2\pi)^{2n}} \zeta(2n), \quad E_{2n} = (-1)^n \frac{4^{n+1} (2n)!}{\pi^{2n+1}} \beta(2n+1)
\end{equation}
where $\beta(s) = L(s,\chi_{4.3})$ is
the Dirichlet beta function,
corresponding to the character modulo $q = 4$ with $(\chi_{4.3}(n))_{n=0}^{\infty} = (0, 1, 0, -1, \ldots)$.

There are $\Theta(n \log n)$ bits in $E_n$ and in the numerator of $B_n$ (the denominator is easy to determine),
so we can recover the exact values by evaluating the $L$-functions
numerically to $p = n^{1+o(1)}$ bits.
As a corollary of Theorem~\ref{thm:complexity}, we have the following:

\begin{theorem}
The $n$th Bernoulli number $B_n$ and Euler number $E_n$ can be computed exactly in time $n^{3/2+o(1)}$ using $n^{1+o(1)}$ space.
\end{theorem}

The significance of this result is that all methods known until quite recently
(for instance those employing the Euler product)
require at least $n^{2+o(1)}$ time.

Harvey~\cite{Harvey2014} gave the first
subquadratic algorithm for computing $B_n$,
which uses $n^{4/3+o(1)}$ time and space
or $n^{3/2+o(1)}$ time when confined to $n^{1+o(1)}$ space.
We fail to improve on Harvey's bound,
but the methods are independent: ours is numerical;
Harvey's uses modular arithmetic and does not involve $L$-functions.
For Euler numbers, no subquadratic algorithm has been published before ours,
though it is plausible that Harvey's algorithm can be generalized to this case.

\subsection{Sparse zeta-expansions}

Many slowly converging series or products can
be evaluated to high precision using \emph{zeta function acceleration}~\cite{flajolet1996zeta},
which is based on the formal rearrangement
\be
\sum_{n \in A} f(1/n) = \sum_m f_m \zeta_A(m), \quad f(z) = \sum_m f_m z^m, \; \zeta_A(s) = \sum_{n \in A} n^{-s}.
\ee
The complexity of approximating such a sum to $p$-bit accuracy
is typically $p^{2+o(1)}$ provided that the transformed series
converges geometrically
and that $O(p)$ coefficients $f_m$ and zeta values $\zeta_A(m)$
can be evaluated simultaneously to $p$-bit accuracy in time $p^{2+o(1)}$ using
FFT-based power series operations (this is the case if $f$ is elementary
and $\zeta_A(s) = \zeta(s)$, for example).

If the resulting zeta-expansion is sparse, then Theorem~\ref{thm:complexity} may yield
an improved complexity bound. An example is the
Landau-Ramanujan constant
\be
\lambda = \left(\frac{1}{2} \prod_{p \equiv 3 \bmod 4} \frac{1}{1-p^{-2}} \right)^{1/2} \approx 0.764
\ee
which appears in the asymptotic formula $\lambda x / \sqrt{\log(x)}$ for the number of integers $k \le x$
expressible as a sum of two squares.
Flajolet and Vardi~\cite{flajolet1996zeta} obtain the sparse zeta-type expansion
\be
\lambda = \frac{1}{\sqrt{2}} \prod_{n=1}^{\infty} \left[ \left(1 - \frac{1}{2^{2^n}}\right) \frac{\zeta(2^n)}{\beta(2^n)} \right]^{1/2^{n+1}}
\ee
which requires only $O(\log(p))$ terms for $p$-bit accuracy.
It follows from Theorem~\ref{thm:complexity} that we can compute $\lambda$ to $p$-bit accuracy in time $p^{3/2+o(1)}$.


\subsection{Stieltjes constants}


The Stieltjes constants $\gamma_k$ are, up to a scaling factor,
the coefficients in the Laurent series of $\zeta(s)$ at $s = 1$.
Theorem~\ref{thm:complexity} states that we can compute any
Stieltjes constant to $p$-bit accuracy in time $p^{3/2+o(1)}$,
which again is superior to the $p^{2+o(1)}$ complexity
of classical methods like Euler-Maclaurin summation~\cite{liang1972stieltjes,Johansson2014hurwitz}
as well as methods based on numerical integration~\cite{JohanssonBlagouchine2018stieltjes}.
Algorithms with $p^{1+o(1)}$ complexity are only known
for Euler's constant $\gamma = \gamma_0$~\cite{BrentMcMillan1980}.

The idea of using \eqref{eq:fe} to compute Stieltjes constants is of course not new.
Coffey \cite[Proposition 9]{coffey2014series} gives the
explicit formula
\be
\gamma = \log(4 \pi) - 2 + \frac{2}{\sqrt{\pi}} \sum_{n=1}^{\infty} \frac{1}{n} \Gamma(\tfrac{1}{2}, \pi n^2) + 2 \sum_{n=1}^{\infty} \Gamma(0, \pi n^2)
\ee
along with a much more complex formula for $\gamma_1$
written in terms of series of ${}_2F_2$ and ${}_3F_3$
hypergeometric functions and digamma functions.
Coffey notes that these formulas ``may have some attraction for computation''
due to the $e^{-\pi n^2}$ type decrease of the terms.
He also notes that the method can be generalized
to Dirichlet $L$-function analogs of Stieltjes constants.
As indicated in the proof of Theorem~\ref{thm:complexity},
implementing~\eqref{eq:fe} with
power series arithmetic provides such
a generalization without requiring the derivation
of unwieldy formulas for the higher derivatives.

Similarly, we obtain a $p^{3/2+o(1)}$ complexity algorithm
for the Glaisher–Kinkelin constant $e^{1/2-\zeta'(-1)}$,
Keiper-Li coefficients, etc.

The approximate functional equation has been used
in a somewhat different way by Keiper~\cite{Keiper1992power} 
to compute Stieltjes
constants and other series coefficients related
to the Riemann zeta function, with worse
complexity than Theorem~\ref{thm:complexity}; we revisit this topic
in \S \ref{sect:transcendental}.

\subsection{Values at rational points}


Theorem~\ref{thm:complexity} implies $p^{3/2+o(1)}$ complexity for computing
the values $L(s, \chi)$ with $s \in \QQ$.
These constants have various applications;
$\zeta(n+1/2)$ and $\beta(n+1/2)$ appear in thermodynamics (Bose-Einstein statistics)
and in connection
with lattice sums describing the electrostatic potentials in crystals (Madelung constants) \cite[\S 1.10]{finch2003mathematical}.
It is a famous open problem whether $L(1/2, \chi) \ne 0$ for all primitive characters $\chi$~\cite[\S 7.6]{platt2011computing}.

The number $\zeta(1/2) \approx -1.4603545$ makes an interesting appearance
in a formula in Ramanujan's lost notebook (see \cite[\S 8.3]{Andrews2013}, where
generalizations to other values of $\zeta(s)$ and $L(s,\chi)$ with $s \in \QQ$ are discussed as well).
For any $\alpha, \beta > 0$ such that $\alpha \beta = 4 \pi^3$,
\begin{equation}
\sum_{n=1}^{\infty} \frac{1}{e^{n^2 \alpha} - 1} = \frac{\pi^2}{6\alpha} + \frac{1}{4} + \frac{\sqrt{\beta}}{4 \pi} \left[ \zeta(\tfrac{1}{2})
+ \sum_{n=1}^{\infty} \frac{\cos(\sqrt{n \beta}) - \sin(\sqrt{n \beta})  - e^{-\sqrt{n \beta}}}{\sqrt{n} (\cosh(\sqrt{n \beta}) - \cos(\sqrt{n \beta}) } \right].
\end{equation}

There is a parallel to the free parameter in \eqref{eq:fe}: by varying $\alpha$, we can force
either the left or the right series to converge faster at the expense of the other.
Setting $\alpha = O(1/p)$ in Ramanujan's formula minimizes the total number of terms
for a given precision $p$;
this leads to an algorithm with $p^{2+o(1)}$ bit complexity to compute $\zeta(1/2)$,
comparable to Euler-Maclaurin summation and inferior to Theorem~\ref{thm:complexity}.

\subsection{Hurwitz zeta-type functions}



The method behind Theorem~\ref{thm:complexity} is not restricted to ``proper'' $L$-functions.
Crandall \cite{crandall2012unified,bailey2015crandall} has
given a formula analogous to \eqref{eq:fe} (Crandall calls this
a ``Riemann-splitting representation'')
for the Lerch
transcendent, which is the analytic continuation of the series
\be
\Phi(z,s,a) = \sum_{n=0}^{\infty} \frac{z^n}{(n+a)^s}, \quad |z| < 1.
\ee

Combining Crandall's expansion with bit-burst
evaluation of the incomplete gamma function should
lead to $p^{3/2+o(1)}$ algorithms for the following:
\begin{itemize}
\item The Lerch transcendent $\Phi(z,s,a)$ with $s \in \QQbar$, $z, a \in \CC$ and its $s$-derivatives.
\item The Hurwitz zeta function $\zeta(s,a)$ with $s \in \QQbar, a \in \CC$ and its $s$-derivatives.
\item The generalized Stieltjes constants $\gamma_n(a)$ with $a \in \CC$.
\item The polylogarithm $\operatorname{Li}_s(z)$ with $s \in \QQbar, z \in \CC$ and its $s$-derivatives.
\item The Barnes $G$-function $G(z)$ with $z \in \CC$.
\end{itemize}

We have not checked the details of these computations
(validity of analytic continuations, possible exceptional points,
explicit error bounds, uniform complexity with respect to parameters),
and we leave this for a future study.

As in the case of integer zeta values,
$p^{1+o(1)}$ algorithms are already available
for the above functions in some more restricted cases,
e.g.\ for $\zeta(n,a)$ with $a \in \QQ$.

\section{The algorithm}

\label{sect:algorithm}

Since the proof of Theorem~\ref{thm:complexity} above is quite terse
and the bit-burst algorithm 
for generic holonomic functions
requires much more complicated machinery
than in the specialized case of computing $\Gamma(a,z)$,
we give a more explicit description here.

We may rely on ball arithmetic~\cite{vdH:ball,Joh2017},
which means that explicit error bounds need to be derived
only for the truncation errors in infinite series;
asymptotic estimates suffice for choosing the floating-point precision.

\subsection{The outer series}

To bound the tails of the infinite series in \eqref{eq:fe},
the following formulas may be used.
Similar bounds can also be found in~\cite{rubinstein1998evidence}.

\begin{lemma}
For real $z > 0$ and complex $a = \sigma + \tau i$,
the order $j \ge 0$ parameter derivative of the incomplete gamma function $\Gamma(a,z)$ satisfies the bound
\be
\begin{aligned}
\label{eq:incgambound}
\left|\Gamma^{(j,0)}(a,z)\right|
&= \left|z^{a-1} \log^j(z) e^{-z} \int_0^{\infty} e^{-t} \left(1+\frac{t}{z}\right)^{a-1} \left(1+\frac{\log(1+t/z)}{\log(z)} \right)^j dt \right| \\
&\le z^{\sigma-1} \log^j(z) e^{-z} \int_0^{\infty} \exp\left(-t + \frac{(\sigma-1) t}{z} + \frac{j t}{z \log(z)}\right) dt \\
&\le \frac{z^{\sigma-1} \log^j(z) e^{-z}}{1-B_j(\sigma,z)}, \quad  B_j(\sigma,z) = \frac{1}{z}\left(\max(\sigma-1, 0) + \frac{j}{\log(z)}\right) \\
\end{aligned}
\ee
provided that $B_j(\sigma,z) < 1$, and assuming that $z > 1$ if $j \ge 1$.
\end{lemma}


We illustrate how to bound the zeroth derivative
of the first of the two infinite series in \eqref{eq:fe}:
\begin{lemma}
Assume that $N, q \ge 1$, $s = \sigma + \tau i$ with $\sigma, \tau \in \RR$, $\delta \in \{0,1\}$.
Define $C = (\sigma + \delta)/2$ and $D = \pi \alpha / q$.
If $D N^2 > C - 1$, then
\begin{equation}
\label{eq:febound}
\begin{aligned}
\left|\sum_{n=N}^{\infty} \frac{\chi(n)}{n^s} \, \Gamma\!\left(\frac{s+\delta}{2}, \frac{\pi n^2 \alpha}{q}\right) \right|
& \le \sum_{n=N}^{\infty} \frac{1}{n^{\sigma}} \frac{(D n^2)^{C-1} e^{-D n^2}}{1-B_0(C, D n^2)} \\
& = \frac{D^{C-1}}{1-B_0(C, D N^2)} \sum_{n=N}^{\infty} \frac{e^{-D n^2}}{n^{2-\delta}} \\
& \le \frac{D^{C-1}}{1-B_0(C, D N^2)} \frac{e^{-D N^2}}{N^{2-\delta} (1-e^{-D})}.
\end{aligned}
\end{equation}
\end{lemma}

We can bound the tails of the $s$-derivatives
as follows: we expand
the power series product $n^{-(s+X)} \Gamma((s+X+\delta)/2, \pi \alpha n^2/q)$ symbolically,
apply the bound \eqref{eq:incgambound} for each coefficient, and
compute geometric series bounds similar to those in \eqref{eq:febound}.

The bound for the other series in \eqref{eq:fe} is identical but
with $C = (1 - \sigma + \delta)/2$ and $D = \pi / (q \alpha)$.

\subsubsection{Implementation remarks}

There is no need to derive a closed formula for choosing $N$;
we can simply evaluate the bound \eqref{eq:febound} for $N = 1, 2, \ldots$ and stop
when the error meets a target tolerance.

For optimal performance, we should compute a tight estimate of the number of bits
that each term contributes to the final result and only compute
to that precision locally.
It is useful to note that for $a \in \RR$ and $z > 0$,
\be
\log(\Gamma(a,z)) \approx \begin{cases} (a-1) \log(z) - z & a < z \\ a (\log(a) - 1) & a \ge z. \end{cases}
\ee
gives an accurate order-of-magnitude
estimate of the incomplete gamma function.




\subsection{Evaluation of the incomplete gamma function}

The idea of the bit-burst algorithm is to analytically
continue a holonomic function $y$ using the Taylor series method for
ODEs, following a path
\be
z_{\text{initial}} \rightsquigarrow z_1 \rightsquigarrow z_2 \rightsquigarrow \ldots
\ee
that approaches the target point $z$ \emph{exponentially} and thus converges
in $O(\log(p))$ steps.
For example, for our application we may choose
$z_1 = \lfloor z / 2^{32}\rfloor 2^{32}$,
$z_2 = \lfloor z / 2^{64}\rfloor 2^{64}$,
$z_3 = \lfloor z / 2^{128}\rfloor 2^{128}$,
$\ldots$, where successive steps double the number
of leading bits extracted from the binary expansion of $z$.\footnote{The initial number of bits
is a tuning parameter; instead of the constant 32,
we may start with $O(\log p)$ bits, for example.}
At each step, the Taylor series can be evaluated using binary splitting,
and the exponentially converging steps balance the bit sizes of the coefficients
against the number of terms in each Taylor series so that the
overall bit complexity is quasilinear in $p$.

\subsubsection{Hypergeometric series}

Let $(a)_n = a (a+1) \cdots (a+n-1)$.
For the first Taylor step $z_{\text{initial}} \rightsquigarrow z_1$,
we may choose $z_{\text{initial}} = 0$. Here we have the hypergeometric series
\be
\Gamma(a,x) = \Gamma(a) - \frac{x^a e^{-x}}{a} \sum_{n=0}^{\infty} \frac{x^n}{(a+1)_n}, \quad x = z_1,
\label{eq:hypser1}
\ee
which
is valid for all $x > 0$ when $a \not \in \{0, -1, -2, \ldots\}$.
If the series is truncated after $N$ terms where $N > -\operatorname{Re}(a) - 1$ and $|a+N+1| > |x|$,
then
\be
\left|\sum_{n=N}^{\infty} \frac{x^n}{(a+1)_n} \right| \le \frac{|x|^N}{|(a+1)_N|} \frac{1}{1-C}, \quad C = \frac{|x|}{|a+N+1|}
\label{eq:ser1bound}
\ee

Truncation bounds for derivatives of this series
with respect to $a$ can
be obtained similarly; see \cite[Theorem 1]{Johansson2019hypergeometric}.

At the poles of the gamma function, a limit computation is needed;
this can be done using power series arithmetic or explicitly using the formula~\cite[8.4.15]{DLMF}
\be
\Gamma(-n,x) = \frac{(-1)^n}{n!} (\psi(n+1) - \log(x)) - x^{-n} \left(\sum_{k=0}^{n-1} + \sum_{k=n+1}^{\infty}\right) \frac{(-x)^k}{k! (k-n)}.
\label{eq:singular}
\ee

When $x = z_1 \approx z$ is sufficiently large, we can also start from
$z_{\text{initial}} = \infty$ and use the asymptotic series
\be
\Gamma(a,x) = x^{a-1} e^{-x} \left[ \sum_{n=0}^{N-1} \frac{(1-a)_n}{(-x)^n} + R_N(a,x) \right]
\label{eq:asympseries}
\ee
for the first step.
We do not need \eqref{eq:asympseries} in the proof of Theorem~\ref{thm:complexity},
but practically speaking it makes a significant difference
for efficiency to choose this expansion whenever
$\min_N |R_N(a,x)|$ is smaller than the target tolerance.
The error term satisfies $|R_N(a,x)| \le |(1-a)_N| / |x|^N$ if $a \in \RR$
and $x > 0$ provided that $N \ge a - 1$. For error bounds with complex $a$,
the formulas in \cite[\S 13.7]{DLMF} may be used.

\subsubsection{Expansions at generic points}

For consecutive Taylor steps $z_{k-1} \rightsquigarrow z_{k}$,
we write the expansion as
\be
\Gamma(a,z_k) = \sum_{n=0}^{\infty} c_n(z_{k-1}) x^n, \quad x = z_k - z_{k-1}.
\label{eq:generictaylor}
\ee
where the coefficients $c_n(z_{k-1})$ need to be determined.
We denote the local
expansion point by $u$ instead of $z_{k-1}$ below to simplify the formulas.

The function $y(z) = \Gamma(a,z)$ satisfies the
second-order
differential equation $z y'' + (z-a+1) y = 0$.
We can translate this differential equation to the point $u$
and convert it to the second-order linear recurrence relation
\be
u (n+1)(n+2) c_{n+2}(u) + (n+1) (n+1+u-a) c_{n+1}(u) + n c_n(u) = 0.
\ee

We can also apply this recurrence to parameter derivatives.
Explicitly, define
\be
c_{j,n}(u) = \frac{\Gamma^{(j,n)}(a,u)}{j! n!}
\ee
so that $\Gamma^{(j,0)}(a,z_k) = \sum_{n=0}^{\infty} c_{j,n}(u) x^n$,
and let $S_n = \sum_{j=0}^J c_{j,n}(u) X^j$ in the ring
of truncated formal power series $\CC[[X]] / \langle X^{J+1} \rangle$.
Then
\be
u (n+1)(n+2) S_{n+2} + (n+1) (n+1+u-a-X) S_{n+1} + n S_n = 0.
\ee

Truncation bounds for the Taylor series \eqref{eq:generictaylor}
and its parameter derivatives
can be obtained using the Cauchy integral formula:
for $n \ge 1$,
\be
c_{j,n}(u) = -\left(\frac{d}{du}\right)^{n-1} \frac{u^{a-1} \log^j(u) e^{-u}}{j! n!} = -\frac{1}{2 \pi i} \frac{1}{j! n} \int_{\gamma} \frac{t^{a-1} \log^j(t) e^{-t}}{(t-u)^n} dt.
\ee

\begin{lemma}
For $a \in \CC$, $u > 0$, $j \ge 0$, $n \ge 1$ and $0 < R < u$, the coefficients $c_{j,n}(u)$ (and $c_n(u) = c_{0,n}(u)$) satisfy
the bound
\be
|c_{j,n}(u)| \le \frac{1}{j! n} \frac{1}{R^{n-1}} M_R(u), \quad M_R(u) = \max_{t : |t-u| = R} |t^{a-1} \log^j(t) e^{-t}|.
\ee
Consequently, for $N \ge 1$ and $|x| < R$, tails of the Taylor series satisfy
\be
\left| \sum_{n=N}^{\infty} c_{j,n}(u) x^n \right| \le \frac{R M_R(u)}{j! N} \frac{C^N}{1 - C}, \quad C = \frac{|x|}{R}.
\ee
\end{lemma}

To bound $M_R(u)$,
it suffices to bound $t^{a-1} \log^j(t) e^{-t}$ on the disk $|t-u| \le R$
using naive upper bounds for the elementary functions,
or using interval arithmetic.
A simple algorithm to choose $R$ is to start with $R = (u+|x|)/2$
and iterate $R \gets R / 2$ as long as this decreases the bound.
The results can be improved slightly with a proper numerical minimization algorithm.






\subsubsection{Overall algorithm}

The final step is to rewrite the series expansions
in matrix form and evaluate the matrix products
using binary splitting.
We sketch the complete algorithm to compute the incomplete gamma function.

We write $x + [\pm \varepsilon]$ to express
the use of an enclosure with midpoint $x$ and radius $\varepsilon$
to represent an exact quantity.

\begin{algorithm}[Bit-burst evaluation of $\Gamma(a,z)$, with $z > 0$]{\ \\} 
\label{alg:bsplit1}
\renewcommand*{\arraystretch}{1.2}
\vspace{-0.8em}
\begin{itemize}
\item Choose initial number of bits $b_1 = 32$ and let $x = z_1 = 2^{b_1} \lfloor z / 2^{b_1} \rfloor$.
\item If the asymptotic series \eqref{eq:asympseries} is accurate enough:
\begin{itemize}
\item Choose $N$ and denote by $\varepsilon$ a bound for the remainder term in \eqref{eq:asympseries}.
\item Compute $P = U_{N-1} \cdots U_1 U_0$ using binary splitting, where $$U_n = \begin{pmatrix} \frac{a-n-1}{x} & 0 \\ 1 & 1 \end{pmatrix}.$$
\item Now $P_{2,1} = \sum_{n=0}^{N-1} \frac{(1-a)_n}{(-x)^n}$.
      Compute $y_1 = x^{a-1} e^{-x} (P_{2,1} + [\pm \varepsilon])$, which equals $\Gamma(a, z_1)$.
\end{itemize}
\item Otherwise:
\begin{itemize}
\item Choose $N$ and denote by $\varepsilon$ a bound for the remainder term in \eqref{eq:hypser1}.
\item Compute $P = U_{N-1} \cdots U_1 U_0$ using binary splitting, where $$U_n = \begin{pmatrix} \frac{x}{a+n+1} & 0 \\ 1 & 1 \end{pmatrix}.$$
\item Now $P_{2,1} = \sum_{n=0}^{N-1} \frac{x^n}{(a+1)_n}$.
      Compute $y_1 = \Gamma(a) - x^a e^{-x} (P_{2,1} + [\pm \varepsilon]) / a$, which equals $\Gamma(a, z_1)$.
\item (At a pole of the gamma function, perform the formal limit computation in the above steps.)
\end{itemize}
\item For $k = 2, 3, \ldots$ with $b_k = 2 b_1$, perform the following:
\begin{itemize}
\item Let $z_k = 2^{b_k} \lfloor z / 2^{b_k} \rfloor$. If this approximates $z$ to within the target precision, set $z_k = z$ instead and make this the last iteration.
\item Compute $x = z_k - z_{k-1}$.
\item Choose $N$ and denote by $\varepsilon$ a bound for the remainder term in \eqref{eq:generictaylor}.
\item Compute $P = U_{N-1} \cdots U_1 U_0$ using binary splitting, where
$$U_n = \begin{pmatrix}
0 & x & 0 \\
\frac{x n}{Q} & \frac{x (n+1) (n+1+z_{k-1}-a)}{Q} & 0 \\
1 & 0 & 1
\end{pmatrix}, \quad Q = -z_{k-1} (n+1) (n+2).$$
\item Compute $y'_{k-1} = -z_{k-1}^{a-1} e^{-z_{k-1}}$.
\item Compute $y_k = P_{3,1} y_{k-1} + P_{3,2} y'_{k-1} + [\pm \varepsilon]$, which equals $\Gamma(a,z_k)$.
\end{itemize}
\item Return $y_k$, which equals $\Gamma(a,z)$.
\end{itemize}
\end{algorithm}

As noted previously, derivatives up to order $j$
with respect to $a$
can be computed using the same algorithm by substituting $a \to a + X$ and
working in $\CC[[X]] / \langle X^{j+1} \rangle$.


Let us elaborate on the technical details in the proof of
Theorem~\ref{thm:complexity}.
When computing the values $\Gamma((s+\delta)/2, z)$ and $\Gamma((1-s+\delta)/2, z)$
used in the approximate functional equation,
the recurrence matrices in Algorithm~\ref{alg:bsplit1}
will have entries in the number field $\mathbb{Q}(s)$,
or in the power series ring $R = \mathbb{Q}(s)[[X]] / \langle X^{j+1} \rangle$
if we compute derivatives.
When both $j$ and the degree of $s$ are fixed, this ring is a finite-dimensional vector space over $\mathbb{Q}$,
and if the minimal polynomial of $s$ has height $h$,
the product or sum of $N$ entries in $R$ with $b$-bit coefficients
will have coefficients with $O(N (b + \log h))$-bit numerators and denominators.

Summing over all bit sizes in the recursion trees
for the binary splitting and the consecutive bit-burst steps
and using the 
bound $b^{1+o(1)}$ for the bit complexity of arithmetic on $b$-bit rational numbers,
we obtain
the $p^{1+o(1)}$ complexity bound for each call to Algorithm~\ref{alg:bsplit1}.

\subsubsection{Implementarion remarks}

In practice, we should clear denominators so that the matrices
have integral entries in the binary splitting products.
The products should then be computed using with truncation (rounding)
to reduce memory usage and improve performance~\cite{mezzarobba2012note}.
Further constant-factor savings are possible by eliminating
various redundant computations in the
binary splitting process.

The gamma function $\Gamma(a)$ with algebraic $a$ can be computed in quasilinear
time by evaluating $\Gamma(a,N)$ with a sufficiently large $N$
using binary splitting~\cite{Brent1976,johansson2021arbitrary}.
In any case, this only needs to be done once: the same $\Gamma(a)$
value can be recycled for all evaluations of $\Gamma(a,z)$.

When using \eqref{eq:hypser1}, there can be significant
cancellation between the gamma function and the series.
This does not affect the \emph{absolute} error when $s$
is small, but when $s$ is large, we need to increase the working
precision to compensate. The precision increases with $z$
up to the point where we can switch to the asymptotic series
\eqref{eq:asympseries}.

In the pseudocode for Algorithm~\ref{alg:bsplit1}, we evaluate the first derivative
$y'(z) = -z^{a-1} e^{-z}$ in each Taylor step. There are several ways
to do this: we can compute the elementary functions from scratch,
we can perform bit-burst analytic continuation
of the function $y'(z)$, or we can perform bit-burst evaluation
of the factors $z^{a-1}$ and $e^{-z}$ separately
using the standard binomial and exponential Taylor series.
Which method performs better may depend on several factors,
but either approach achieves quasilinear complexity.

In the approximate functional equation, we need to evaluate $\Gamma(a,z)$
for successive values $z = C n^2$.
It is tempting to reuse the computed $\Gamma(a,z)$ values,
starting the bit-burst evaluation at $z_{\text{initial}} = C (n-1)^2$.
However, this appears to be a net slowdown,
the main reason being that the recurrence matrices
are much simpler for the hypergeometric series at the origin
than for the expansions at generic points.

\section{Implementation results}

\label{sect:implementation}

We have implemented the
algorithm for $L(s,\chi)$ with $s \in \QQ$ in Arb~\cite{Joh2017}.
We leave an implementation for $s \in \overline{\QQ}$ and for $L^{(j)}(s,\chi)$ for future work.

\subsection{Fixed rational points}

\begin{table}
\setlength{\tabcolsep}{3pt}
\renewcommand{\arraystretch}{1.02}
\centering
\caption{\small
Time in seconds to compute values of $L$-functions at fixed simple rational points, using Euler-Maclaurin summation (EM) and the approximate functional equation (AFE).}
\label{tab:timeconst}
\small
\begin{tabular}{c | l l l | l l l}
 Digits                     & Number       & EM        & AFE     & Number & EM & AFE \\ \hline
 $10^3$                     & $\zeta(1/2)$ & 0.0076    & 0.037   &  $L(1/2, \chi_{23.19})$      & 0.15      & 0.18 \\
 $\lfloor 10^{3.5} \rfloor$ &              & 0.19      & 0.29    &        & 2.3       & 1.5  \\
 $10^4$                     &              & 2.7       & 2.7        &        & 38        & 14   \\
 $\lfloor 10^{4.5} \rfloor$ &              & 52        & 27         &        & 621       & 131    \\
 $10^{5}$                   &              & 887       & 262        &        &    & 1282   \\
 $\lfloor 10^{5.5} \rfloor$ &              &           & 2175       &        &    &    \\
 $10^{6}$                    &              &          & 17004     &        &    &    \\ \hline
 $10^3$                      & $\zeta(4/3)$ & 0.014    & 0.083 &  $L(4/3, \chi_{23.19})$  & 1.5       & 0.38 \\
 $\lfloor 10^{3.5} \rfloor$  & & 0.34      & 0.68      &     & 39        & 3.5  \\
 $10^4$                      & & 6.0       & 7.2       &  & 795       & 30   \\
 $\lfloor 10^{4.5} \rfloor$  & & 100       & 66        &      &           & 295    \\
 $10^{5}$                    & & 1808      & 618       &       &           & 2821   \\
 $\lfloor 10^{5.5} \rfloor$  & &           & 5242      &                     \\
\end{tabular}
\end{table}

Table~\ref{tab:timeconst} illustrates the precision-dependent scaling
of the implementation of the approximate functional equation (AFE) when $s$ is a fixed simple fraction.
We also show timings for the Euler-Maclaurin implementation
of Dirichlet $L$-functions in Arb (EM).\footnote{The benchmarks were run on a 1.90 GHz Intel i5-4300U CPU.}

We observe that the AFE
is competitive from about $10^4$ digits for computing the Riemann zeta function.
The advantage is greater for $L$-functions with
larger modulus $q$ due to the $O(q^{1/2})$ scaling.

At high enough precision, the subquadratic asymptotic complexity
of the AFE is evident
since the measured time increases by (barely) less than a factor 10
when the precision is multiplied by $10^{1/2}$.
Asymptotically for a $p^{3/2+o(1)}$ complexity algorithm,
the time should only increase by
a factor $10^{3/4} \approx 5.6$,
but the tested precisions are small enough
for the hidden logarithmic factors in the complexity bounds
to influence the running time.
The timings for the $p^{2+o(1)}$ EM
algorithm also increase by factors somewhat larger than 10
for the same reason.

There is roughly a factor two slowdown with both algorithms
going from $\zeta(1/2)$ to $\zeta(4/3)$, for different reasons: the AFE is inherently twice
as fast for $\zeta(1/2)$ since only one of the two series has to be computed;
the slowdown with EM is an implementation artifact
(the Arb code is not optimized for rational powers).

The million-digit computation of $\zeta(1/2)$, which takes less than five hours
on a single core and requires negligible memory,
appears to be a precision record for a zeta constant not at an integer.\footnote{In 2013, the author computed the
first nontrivial zero $\tfrac{1}{2} + 14.134\ldots i$ of $\zeta(s)$ to
303,000 digits using Euler-Maclaurin summation. This took 20 hours and used 62 GB of memory,
the high memory usage being the main obstacle to reaching higher precision~\cite{Johansson2014hurwitz} (this
is an implementation problem that can be avoided).}


\subsection{Computation of Bernoulli numbers}

\begin{table}
\setlength{\tabcolsep}{10pt}
\renewcommand{\arraystretch}{1.08}
\centering
\caption{\small
Time in seconds to compute the Bernoulli number $B_n$ and Euler number $E_n$ using Harvey's multimodular algorithm (MM), the Euler product (EP), and the approximate
functional equation (AFE). Timings marked * were estimated
based on the time to evaluate a sparse subsequence ($1/10$ or $1/100$) of the terms,
giving an accurate estimate of the time for the full computation without performing it. We indicate the number of digits in
the numerator of $B_n$ and in $E_n$.}
\label{tab:timebernoulli}
\small
\begin{tabular}{l c l l l l}
Number   & $n$                           & Digits & MM       & EP            & AFE \\ \hline
$B_n$    &  $10^3$                       & 1779          & 0.0066   & 0.00010       & 0.067    \\
         &  $\lfloor 10^{3.5} \rfloor$   & 7180          & 0.025    & 0.0011        & 0.83     \\
         &  $10^4$                       & 27691        & 0.10     & 0.012         & 11       \\
         &  $\lfloor 10^{4.5} \rfloor$   & 103330        & 0.47     & 0.18          & 142      \\
         &  $10^{5}$                     & 376772       & 2.7      & 1.9           & 1707     \\
         &  $\lceil 10^{5.5} \rceil$     & 1349518       & 22       & 21            & 16578    \\
         &  $10^{6}$                     & 4767554                & 222      & 224           & 159945*  \\
         &  $\lceil 10^{6.5} \rceil$     & 16657389                & 2329     & 2567          & 1587800* \\ \hline
$E_n$    &  $10^3$                       & 2372       &          & 0.00026       & 0.19     \\
         &  $\lfloor 10^{3.5} \rfloor$   & 9076       &          & 0.0026        & 2.0      \\
         &  $10^4$                       & 33699       &          & 0.033         & 24       \\
         &  $\lfloor 10^{4.5} \rfloor$   & 122367       &          & 0.49          & 293      \\
         &  $10^{5}$                     & 436962       &          & 5.9           & 2874     \\
         &  $\lceil 10^{5.5} \rceil$     & 1539903       &          & 68            &          \\
         &  $10^{6}$                     & 5369590  &          & 726           &          \\
\end{tabular}
\end{table}

Table~\ref{tab:timebernoulli}
compares three algorithms to compute $B_n$ as an exact fraction:
\begin{itemize}
\item MM: Harvey's implementation of his $n^{2+o(1)}$ multimodular algorithm~\cite{Harvey2010}
(available in the \texttt{bernmm} module in SageMath).
\item EP: the classical $n^{2+o(1)}$ zeta function algorithm
using the Euler product implemented in Arb.
\item AFE: the approximate functional equation implemented in Arb.
\end{itemize}

The last two implementations also support computing Euler numbers.

The MM and EP algorithms both scale
superquadratically with $n$, the multimodular algorithm having a slight edge for $n$ larger
than $10^4$.
This confirms the observations in \cite[Table~1]{Harvey2010}.
The AFE appears to scale weakly subquadratically, but for reasonably sized $n$,
it is roughly $10^3$ times slower than the Euler product.
To explain this gap we need to consider the logarithmic
and constant factor overheads that we have neglected
in the complexity analysis so far.

An analysis with Stirling's formula shows that the cutoff in the Euler product for computing $\zeta(n)$ to $\log_2 |B_n|$ bits of accuracy is $N \approx n / (2 \pi e)$,
and there are about $N / \log N$ primes up to this cutoff.
A similar analysis for the AFE gives the cutoff $N \approx n^{1/2} (2 \pi)^{-1/2} \log(n/(2 e))$, where all terms are needed, and there are two such series to compute.
Considering these facts alone, the AFE thus saves at most a factor $n^{1/2} / (2 e \sqrt{2 \pi} \log^2 n) \approx n^{1/2} / (13.6 \log^2 n)$ asymptotically,
which means we need $n \approx 10^7$ to break even \emph{assuming that the terms have unit cost}.

The last assumption is obviously false: series of matrix products are roughly $O(\log^2 n)$
slower than integer powers. An asymptotic speedup of $n^{1/2} / (C \log^4 n)$, $C > 10^1$,
is consistent with the observed three-orders-of-magnitude slowdown for $n \approx 10^6$
and suggests that we may need $n$ larger than $10^{15}$ for the AFE to win, though
there is too much uncertainty to extrapolate reliably.

An interesting question is whether Harvey's subquadratic algorithm for Bernoulli numbers~\cite{Harvey2014}
can perform better, but unfortunately no implementation exists.


\section{Evaluation for non-algebraic $s$}

\label{sect:transcendental}

If $s$ is not algebraic and instead must be represented by a $p$-bit
floating-point approximation, then the bit-burst algorithm does not
offer any improvement over naive series evaluation since the recurrence
matrices will not have small entries,
and we only obtain a $p^{5/2+o(1)}$ algorithm to compute $L(s,\chi)$ via \eqref{eq:fe}.
However, there at least four independent ways to reduce the complexity to $p^{2+o(1)}$:

\begin{enumerate}
\item We evaluate each $\Gamma(a,z)$ via \eqref{eq:hypser1} (optionally together with \eqref{eq:asympseries})
using a baby-step giant technique, exploiting the hypergeometric structure
of the terms: if the truncated series
is represented as a matrix product $\prod_{n=0}^{N-1} U_n$ of length $N = m^2$,
we expand $\prod_{n=0}^{m-1} U_{x+n}$ as a matrix of rational functions in $x$
and evaluate at $m$ points using fast multipoint evaluation~\cite{Borwein1987}.
This achieves $p^{3/2+o(1)}$ complexity for each incomplete gamma function.
This is the approach described in~\cite{BorweinBradleyCrandall2000}.
\item We expand a truncation of the series \eqref{eq:hypser1} (optionally together with \eqref{eq:asympseries})
as a polynomial in $z$ of degree $p^{1+o(1)}$ and evaluate it simultaneously at the requisite
$p^{1/2+o(1)}$ values of $z$ using fast multipoint evaluation.
\item Instead of using the series in incomplete gamma function, we
use the integral form \eqref{eq:fetheta} and its analog for Dirichlet $L$-functions.
Using a standard numerical integration method with geometric
rate of convergence for analytic functions, for example Gaussian,
Clenshaw-Curtis or double exponential quadrature, we need $p^{1+o(1)}$
evaluations of the integrand.
Evaluating the theta function in the integrand using the $q$-series costs $p^{1/2+o(1)}$ multiplications,
resulting in an $p^{5/2+o(1)}$
algorithm for $L(s,\chi)$. This is the method used by Keiper~\cite{Keiper1992power}. However, we can compute the theta function in
quasilinear time using arithmetic-geometric mean iteration instead~\cite{labrande2018computing}, and this achieves $p^{2+o(1)}$ complexity for $L(s,\chi)$.
\item As above, but we expand the truncated $q$-series for the theta
function as a polynomial and evaluate it at all the integration
nodes using fast multipoint evaluation.
\end{enumerate}

The techniques for fast evaluation of theta functions used in methods (3) and (4)
have previously been used in the context of computing class polynomials
via numerical approximations of the roots~\cite{enge2009complexity}.

We also mention a version of method (2) that is
asymptotically slower but may be superior at realistic levels of precision.
We can expand a truncation of the series \eqref{eq:hypser1} (and optionally \eqref{eq:asympseries})
for $\Gamma(a, \pi n^2 \alpha / q)$ as a polynomial in $n^2$.
This reduces the computation to multipoint evaluation at
the small integers $n = 1, 2, \ldots$,
which may be performed using repeated applications of
Horner's rule instead of fast multipoint evaluation.
This results in a $p^{5/2+o(1)}$ algorithm
but with very little overhead since the Horner evaluations only
involve additions and $p$-by-1-word multiplications
which are orders of magnitude
cheaper than full $p$-by-$p$ multiplications.

Yet another option is the Booker-Molin method, which employs a Fourier series
that can be precomputed for efficient
evaluation at many values of $s$~\cite[\S 9.4]{belabas2021numerical}.

It is not clear \emph{a priori} which of the above methods will perform better,
so further implementation studies are needed.

Did Borwein and Borwein~\cite{Borwein1988}
have one of the methods above in mind for non-algebraic $s$, or did
they have an entirely different algorithm?
It is likely that they considered (1) or (2) since their paper discusses the
same multipoint evaluation techniques for other functions,
although their wording
is more suggestive of an algorithm
along the lines of (3) or (4).
Either way, there seems to be no obvious way to
obtain a subquadratic algorithm for $\zeta(s)$
or $L(s,\chi)$ for non-algebraic $s$,
and this remains an open problem
along with the problem of finding a quasilinear
algorithm for $s \in \overline{\mathbb{Q}} \setminus \ZZ$.

\section*{Acknowledgements}

The author was supported in part by the ANR grant ANR-20-CE48-0014-02 NuSCAP.

\bibliographystyle{alpha}
\bibliography{references}

\end{document}